\documentclass[11pt]{amsart}
\usepackage{verbatim}
\usepackage{amsmath}
\usepackage{amsrefs}
\usepackage{enumerate}
\usepackage{amssymb}
\usepackage{color}
\usepackage{url}
\usepackage{tikz-cd}
\usepackage{tabu}
\usepackage{array}

\DeclareMathOperator{\Pic}{Pic}

\DeclareMathOperator{\Sym}{Sym}
\DeclareMathOperator{\End}{End}
\DeclareMathOperator{\Mor}{Mor}
\DeclareMathOperator{\Hom}{Hom}
\DeclareMathOperator{\rank}{rank}

\renewcommand{\phi}{\varphi}
\newcommand\+{\;\lower\plusheight\hbox{$+$}\;}
\newcommand\lldots{\;\lower\plusheight\hbox{$\cdots$}\;}

\newtheorem{Theorem}{Theorem}[section]
\newtheorem*{theorem*}{Theorem} 
\newtheorem*{Proposition*}{Proposition} 
\newtheorem{Corollary}[Theorem]{Corollary}
\newtheorem*{Corollary*}{Corollary} 
\newtheorem{Lemma}[Theorem]{Lemma}
\newtheorem{Definition}[Theorem]{Definition}

\newtheorem{Remark}[Theorem]{Remark}
\newtheorem{Proposition}[Theorem]{Proposition}

\setbox0=\hbox{$+$}
\newdimen\plusheight
\plusheight=\ht 0 \setbox0=\hbox{$-$}
\newdimen\minusheight
\minusheight=\ht0

\setbox0=\hbox{$\cdots$}
\newdimen\cdotsheight
\cdotsheight=\plusheight
\newcommand\blankfootnote[1]{%
	\let\thefootnote\relax\footnotetext{#1}%
	\let\thefootnote\svthefootnote%
}

\title{ Vanishing of Hyperelliptic L-functions at the Central Point}
\author{Wanlin Li}
\address{
	Department of Mathematics, University of Wisconsin\\
	480 Lincoln Drive, Madison, Wisconsin, 53706 USA\\
	E-mail: wanlin@math.wisc.edu}

\begin{document}
\blankfootnote{\textcopyright 2018. This manuscript version is made available under the CC-BY-NC-ND 4.0 license \url{http://creativecommons.org/licenses/by-nc-nd/4.0/}}	
	
	\begin{abstract}
	We obtain a lower bound on the number of quadratic Dirichlet L-functions over the rational function field which vanish at the central point $s = 1/2$. This is in contrast with the situation over the rational numbers, where a conjecture of Chowla predicts there should be no such L-functions. The approach is based on the observation that vanishing at the central point can be interpreted geometrically, as the existence of a map to a fixed abelian variety from the hyperelliptic curve associated to the character.
	\end{abstract}
	\maketitle

	\allowdisplaybreaks
	\numberwithin{equation}{section}
	
	\section{Introduction}
	
	S. Chowla conjectured in \cite{Chowla} that, for any real non-principal Dirichlet character $\chi$, $L(s,\chi) \neq 0$ for all $s \in (0,1)$. In particular, his conjecture asserts that L-functions of quadratic characters never vanish at the central point $s = 1/2$.
	
	Although this conjecture is still open, much progress has been made. K. Soundararajan~\cite{Soundararajan} proved that at least $87.5 \% $ of odd squarefree positive integers $d$ have the property $L(1/2,\chi_{8d}) \neq 0$ where $\chi_{8d}$ denotes the quadratic character with conductor $8d$.
	
	In this paper, we consider the analogue of Chowla's conjecture obtained by replacing the field of rational numbers with the field of rational functions over a finite field.
	
	Let $q = p^e$ be a power of an odd prime $p$ and $\mathbb{F}_q$ the finite field with $q$ elements. Let $k=\mathbb{F}_q(t)$ denote the field of rational functions over $\mathbb{F}_q$. The primes of $k$ are represented by monic irreducible polynomials in $\mathbb{F}_q[t]$ except the one prime at infinity.
	
	A quadratic character of $k$ corresponds to a squarefree polynomial in $\mathbb{F}_q[t]$. Explicitly, take $D \in \mathbb{F}_q[t]$ to be a squarefree polynomial and $K = k(\sqrt{D})$ the quadratic extension of $k$ by joining $\sqrt{D}$. Then we can define a quadratic character $\chi_D$ as follows:
	
	For $P$ a prime of $k$,
	\[
	\chi_D(P) = \begin{cases}
	1&  \text{$P$ splits in $K$}\\
	-1& \text{$P$ is inert in $K$}\\
	0&  \text{$P$ ramifies in $K$}\\
	\end{cases}
	\] 

	We define the L-function associated to $\chi_D$ as 
	\[ L(s, \chi_D) = \prod_{P}(1 - \chi_D(P)|P|^{-s})^{-1} \]
	where the product is taken over the primes represented by polynomials $P$ and $|P| = q ^ {deg P}$.
	
	\begin{Definition}\label{DefinePN}
		Define sets: 
		\begin{align*}
		P(N)& = \{ D \in \mathbb{F}_q[t] : D \text{ monic, squarefree, } |D|<N \}\\
		g(N)& = \{ D \in P(N) : L(1/2, \chi_D) = 0 \}.
		\end{align*}
	\end{Definition}

	\begin{Remark}
		\upshape
		Note that in the definition above, we have restricted ourselves to characters corresponding to monic squarefree polynomials which is half of all quadratic characters. But since we only study the density in this paper, such restriction won't affect our results.
	\end{Remark}
		
	Under this definition, the analogue of Chowla's conjecture states that $g(N)$ is empty for any $N$. There are some results towards this statement. 
	
	Bui and Florea~\cite{BF} showed for a fixed finite field $\mathbb{F}_q$ with odd characteristic, as $N \rightarrow \infty $,
	$$ |g(N)| \ll 0.057N + o(1) $$
	where $N=q^{2n+1}$ for some $n>0$.
	
	The purpose of this paper is to show that the analogue of Chowla's conjecture over $\mathbb{F}_q(t)$ is not correct and to give a lower bound on the number of counterexamples with bounded height.
	
	\begin{Theorem}\label{MainThm}
		Let $q=p^e$ and let $g(N)$ be the set defined in Definition \ref{DefinePN}. For any $\epsilon >0$, there exist nonzero constants $B_\epsilon$ and $N_\epsilon$, such that
		\begin{enumerate}
			\item when $e$ is even, $|g(N)|\geq  B_\epsilon \cdot N^{1/2 - \epsilon} $ for $N > N_\epsilon$.
			\item when $e$ is odd and $q \ne 3$, $|g(N)|\geq  B_\epsilon \cdot N^{1/3 - \epsilon} $ for $N > N_\epsilon$.
			\item when $q = 3$,  $|g(N)|\geq  B_\epsilon \cdot N^{1/5 - \epsilon} $ for $N > N_\epsilon$.
		\end{enumerate}
		In particular, as $N \rightarrow \infty$, $|g(N)|$ approaches infinity.
	
	\begin{Remark}
		\upshape{
		Although Chowla's conjecture is not strictly true over $\mathbb{F}_q(t)$, it may hold for \textit{almost all} quadratic characters, i.e. it may be the case that $|g(N)|/N \rightarrow 0$ as $N \rightarrow \infty$.}
	\end{Remark}
	\end{Theorem}

	\textbf{Outline of paper.} In section 2, we give a geometric interpretation for the vanishing of a quadratic L-function at the central point. In section 3, we show a lower bound on the number of hyperelliptic curves which admit a dominant map to some fixed curve. In section 4, we describe an application of our main theorem to give a lower bound on the number of elliptic curves with elevated ranks in certain quadratic twist families. In section 5, we provide the proof of Theorem \ref{MainThm}. In section 6, we present some of the data we collected using Magma on this problem.
	
	\textbf{Acknowledgments.} I would like to thank my advisor Jordan Ellenberg for bringing this problem to me and his guidance during my work. I would like to thank Dima Arinkin and Melanie Matchett Wood for useful conversations on the subject matter of this paper. I want to also thank Alexandra Florea and Solly Parenti for reading the earlier draft and giving valuable feedbacks. This work was partially supported by NSF-DMS grant 1700884.

	\section{Geometric Interpretation of Vanishing at the Central Point}
	
	Let $D$ be a monic squarefree polynomial over $\mathbb{F}_q$. Then $y^2 = D$ is a hyperelliptic curve defined over $\mathbb{F}_q$ which we denote by $C$ from now on. The field $K = k(\sqrt{D})$ as defined before is the function field of $C$.
	
	Let $P(x) \in \mathbb{Z}[x] $ be the characteristic polynomial of geometric Frobenius acting on the Jacobian $J(C)$.
	
	Then we get 
	\[ P(q^{-s})=(1-q^{-s})^{\lambda_D}L(s, \chi_D) \]
	where 
	\[
	\lambda_D = \begin{cases}
	1& \deg D  \text{ even}\\
	0& \deg D \text{ odd}\\
	\end{cases}
	\] 
	
	By the Riemann Hypothesis for curves over finite fields, we have a factorization
	\[ P(x) = \prod_{j=1}^{2g}(1-x \pi_j) \]
	where $g$ is the genus of $C$ and $\pi_j$ an algebraic integer with $|\pi_j| = q^{1/2}$ under every complex embedding.

	The following lemma is now immediate.
	\begin{Lemma}\label{geometry}
	Let $D$ be a monic squarefree polynomial in $\mathbb{F}_q[t]$ and $\chi_D$ be the quadratic character associated to $D$. Let $C$ be the hyperelliptic curve defined by $y^2 = D$, $P \in \mathbb{Z}[x]$ the characteristic polynomial of geometric Frobenius acting on the Jacobian of $C$ and $\pi_1, \ldots , \pi_{2g}$ the eigenvalues of this action. Then the following statements are equivalent: 
	\begin{itemize}
		\item $L(1/2,\chi_D) = 0$.
		\item $P(q^{-1/2}) = 0$.
		\item $\pi_j = q^{1/2} \text{ for some } j$.
	\end{itemize}

	\end{Lemma}

	Algebraic integers with all Archimedean absolute values equal to $q^{1/2}$ are called Weil integers. The theorem of Honda--Tate states every Weil integer is an eigenvalue of the geometric Frobenius acting on some simple abelian variety over $\mathbb{F}_q$.
	
	\begin{Theorem}[Honda--Tate \cite{E}, \cite{Honda}]\label{HondaTate}
		Let $A$ be an abelian variety defined over $\mathbb{F}_q$ and $\pi_A$ an eigenvalue of the geometric Frobenius endomorphism of $A$. The map $A \mapsto \pi_A $ defines a bijection between the $\mathbb{F}_q$-isogeny classes of abelian varieties defined and simple over $\mathbb{F}_q$ and Galois conjugacy classes of Weil integers.
	\end{Theorem}

	In particular, Honda--Tate guarantees the existence and uniqueness of an isogeny class of simple abelian varieties over $\mathbb{F}_q$ with $q^{1/2}$ being an eigenvalue of the Frobenius. We denote a representative of this class by $A_q$.
	
	Now we want to find hyperelliptic curves whose Jacobians have $q^{1/2}$ as a Frobenius eigenvalue. Any curve $C$ with a nonconstant map to $A_q$ has $A_q$ as an isogeny quotient of $J(C)$, which implies $C$ has $q^{1/2}$ as a Frobenius eigenvalue. A theorem of Tate guarantees the converse also holds.
	
	\begin{Theorem}[Tate \cite{E} , \cite{Mumford}]\label{Tate}
		Let $A$ and $B$ be abelian varieties defined over $\mathbb{F}_q$ and let $f_A$, $f_B \in \mathbb{Z}[T]$ be characteristic polynomials of geometric Frobenius on $A$ and $B$. Then the following are equivalent:
		\begin{itemize}
			\item[1)]  $B$ is $\mathbb{F}_q$-isogenous to a sub-abelian variety of $A$;
			\item[2)]  $f_B \mid f_A$ in $\mathbb{Q}[T]$.
		\end{itemize}
	\end{Theorem} 
	
	\begin{Proposition}\label{subvariety}
		Let $C$ be a hyperelliptic curve defined over $\mathbb{F}_q$, then		
		$q^{1/2}$ is an eigenvalue for geometric Frobenius acting on $J(C)$ if and only if $A_q$ is $\mathbb{F}_q$-isogenous to a sub-abelian variety of $J(C)$.
	\end{Proposition}

	\begin{proof}
		By the theorem of Honda--Tate, there is a unique isogeny class of simple abelian varieties over $\mathbb{F}_q$ having $q^{1/2}$ as a Frobenius eigenvalue, i.e. the class containing $A_q$. Since $J(C)$ can be decomposed up to isogeny uniquely as products of simple abelian varieties over $\mathbb{F}_q$, by the theorem of Tate, $q^{1/2}$ being a Frobenius eigenvalue for $J(C)$ is equivalent to $J(C)$ having a simple factor isogenous to $A_q$.
	\end{proof}

	\begin{Proposition}\label{map}
	$L(1/2, \chi_D) = 0$ if and only if the hyperelliptic curve $C: y^2 = D$ admits a nontrivial map to $A_q$.
	\end{Proposition}

	\begin{proof}
		By Lemma \ref{geometry} and Proposition \ref{subvariety}, $L(1/2,\chi_D)=0$ if and only if $A_q$ is $\mathbb{F}_q$-isogenous to a sub-abelian variety of $J(C)$.
		
		Thus, equivalently there is a dominant map $J(C) \rightarrow A_q$ over $\mathbb{F}_q$.
		
		And as long as we have a map $C \rightarrow J(C)$ over $\mathbb{F}_q$ such that the image doesn't lie in any coset of the kernel of the projection to $A_q$, the composition gives a nonconstant morphism from $C$ to $A_q$.
		
		To construct such a map, we just need to take the canonical class $\omega_C$ and define 
		\begin{align*}
		 C& \rightarrow J(C) \\
		 P& \mapsto (2g-2)P-\omega_C 
		\end{align*}

		Then the image of $C \left( \overline{\mathbb{F}}_q \right) $ under this map generates $J(C)$ as a group. Thus it is not contained in the kernel of $J(C) \rightarrow A_q$ and intersect the kernel non-trivially. This shows the existence of a nontrivial map from $C$ to $A_q$.
		
		Conversely, if there exists a nontrivial map from $C$ to $A_q$, it factors through the Albanese variety of $C$ which is the dual abelian variety of $J(C)$. Since Jacobian varieties are self-dual, this induces a nontrivial map from $J(C)$ to $A_q$. Since $A_q$ is $\mathbb{F}_q$-simple, this map is surjective. This implies that the map from $C$ to $A_q$ is surjective as desired.
	\end{proof}

	 Proposition \ref{map} supplies a geometric condition equivalent to our algebraic statement $L(1/2, \chi_D) = 0$. All we need is to use this geometric condition to construct desired polynomials $D$.

	\section{Maps Between Hyperelliptic Curves}
	
	In this section, we prove the following result which provides a lower bound for the number of hyperelliptic curves of bounded genus covering a fixed hyperelliptic curve over a finite field of odd characteristic.
	
	\begin{Proposition}\label{MainProp}
		Let $C_0$ be a hyperelliptic curve of genus $g$ defined over $\mathbb{F}_q$ where $q$ is odd. Assume the existence of a defining equation of $C_0$ as $y^2=f(x)$ where $\deg f = 2g+2$ and $f$ is reducible or $\deg f=2g+1$ and $f$ need not to be reducible.
		Then for any $\epsilon>0$, there exist positive constants $B_\epsilon$ and $N_\epsilon$ such that the number of polynomials $D \in \mathbb{F}_q[t]$ satisfying
		\begin{itemize}
			\item $|D| < N$
			\item Curve $C : s^2 = D(t)$ admits a dominant map to $C_0$
		\end{itemize}
		is at least $B_\epsilon \cdot N^{\frac{1}{g+1} - \epsilon}$ for $N>N_\epsilon$.
	\end{Proposition}

	The restriction on the form of the defining equation of $C_0$ is only used for the proof of Proposition \ref{MainProp}. Lemma \ref{CtoE} and \ref{C1toC2} hold for general hyperelliptic curves.

	The proposition is based on two lemmas relating maps between hyperelliptic curves to maps from $\mathbb{P}^1$ to $\mathbb{P}^1$. The treatment is slightly different when the base curve is an elliptic curve and when the base curve has higher genus, we treat the two cases separately, in Lemma \ref{CtoE} and Lemma \ref{C1toC2} respectively.

	\begin{Lemma}\label{CtoE}
		Let $\phi : C \rightarrow E $ be a dominant map from a hyperelliptic curve to an elliptic curve over a field $k$ where char $k \ne 2$. Let $C/ \iota_C$ be the degree $2$ map from $C$ to $\mathbb{P}^1$ induced by the hyperelliptic involution and $E/ [-1]$ be the degree $2$ map from $E$ to $\mathbb{P}^1$ induced by the elliptic involution. Then there exists a dominant map $\psi :C \rightarrow E  $ together with a map $ h(x) : \mathbb{P}^1 \rightarrow \mathbb{P}^1$ and a point $R \in E(k)$ such that the following diagram commutes. 
		
		\[\begin{tikzcd}
		C \arrow{r}[above]{C/ \iota_C} \arrow{d}{2\phi} \arrow[bend right]{dd}[left]{\psi}
		& \mathbb{P}^1 \arrow[dd,dotted, "h"] \\
		E \arrow[d, "+R"]\\
		E \arrow{r}[above]{E/ [-1]}
		& \mathbb{P}^1
		\end{tikzcd}\]
		
	\end{Lemma}
	
	\begin{proof}
		
		Take any point $P$ on $C$ and denote $\overline{P} = \iota_C(P)$; then $P + \overline{P}$ is linearly equivalent to $P' + \overline{P'}$ for any point $P'$ on $C$.
		
		We have $$\phi(P)+\phi(\overline{P}) = \phi(P') + \phi(\overline{P'}) = R$$ where  $R$ is a $k$-point of $E$.
		
		Define $\psi$ by the rule $\psi(P) = 2\phi(P)-R$. 
		
		Thus $$\psi (P) + \psi (\overline{P}) = 2 \phi(P) - R + 2\phi(\overline{P}) - R =O $$ which means it is equivariant for the two involutions as desired.
		
	\end{proof}

		\begin{Lemma}\label{C1toC2}
		Let $C_1$ and $C_2$ be hyperelliptic curves with genus greater than $1$ over a field $k$ where char $k \ne 2$. Let $\psi : C_1 \rightarrow C_2$ be a dominant map from $C_1$ to $C_2$. Then there exists a rational function $h$ over $k$ such that the following diagram commutes:
		
		\[\begin{tikzcd}
		C_1 \arrow{r}[above]{C_1/ \iota_1} \arrow{d}{\psi}
		& \mathbb{P}^1 \arrow[d,dotted, "h"] \\
		C_2 \arrow{r}[above]{C_2/ \iota_2}
		& \mathbb{P}^1
		\end{tikzcd}\]
		
		where $\iota_1$ and $\iota_2$ are the hyperelliptic involutions on $C_1$ and $C_2$.
	\end{Lemma}
	
	\begin{proof}
	Let $C$ be a hyperelliptic curve of genus greater than $1$ and let $W$ be a Weierstrass point of $C$.Then the fiber over $2W$ in $\Pic^2(C)$ of the natural map $\Sym^2(C) \rightarrow \Pic^2(C)$ is given by divisors of form $P+\overline{P}$ where $\overline{P}$ denotes the image of $P$ under the hyperelliptic involution.
	
	Note that for an elliptic curve, every point is Weierstrass and thus $2W$ doesn't specify a unique divisor class in $\Pic^2(C)$.
	
	Thus considering $\psi$ induces a map from $\Pic^2(C_2)$ to $\Pic^2(C_1)$, a pair of conjugate points on $C_1$ get mapped to a pair of conjugate points on $C_2$.
	\end{proof}
	
	\begin{Proposition}\label{Tofunctionfield}
		There exists a dominant map from a hyperelliptic curve $C: s^2 = D(t)$ where $D(t) \in \mathbb{F}_q[t]$ to a hyperelliptic curve $C_0: y^2 = f(x)$ where $f(x) \in \mathbb{F}_q[x]$ if and only if the quadratic twist $Dy^2 = f(x)$ has a nontrivial rational point $(x_0,y_0)$ over $\mathbb{F}_q(t)$.
	\end{Proposition}
	
	\begin{proof}
	By the previous two lemmas \ref{CtoE} and \ref{C1toC2}, if a dominant map exists, $C$ has a defining equation of the form $ s^2 = f(h(t))$ for some rational function $h(t)$ in $\mathbb{F}_q(t)$. Thus, there exists $p(t) \in \mathbb{F}_q(t)$ such that $$D(p(t))^2 = f(h(t))$$ which is saying $$(x_0,y_0)=(h(t),p(t))$$ satisfies $$Dy_0^2=f(x_0).$$
	
	On the other hand, if there exists a point $(x_0,y_0)$ on the curve defined by $Dy^2 = f(x)$, then we have a dominant map $(t,s) \mapsto (x_0,y_0s)$ from $C$ to $C_0$.
	\end{proof}

	From Proposition \ref{Tofunctionfield}, our question about squarefree polynomials $D \in \mathbb{F}_q[t]$ with curve defined by $s^2=D(t)$ admitting a dominant map to a hyperelliptic curve of genus $g$ defined by $y^2 =f(x)$ is exactly the same as asking for nontrivial solutions of the equation $Dy^2=f(x)$ over the function field $\mathbb{F}_q(t)$. 
	
	Let $$h(t) = u(t) / v(t) \in \mathbb{F}_q(t),$$ where $u(t)$, $v(t) \in \mathbb{F}_q[t]$ and $ p(t) \in \mathbb{F}_q(t)$.
	
	If we have $Dp^2= f(h)$, then we get $$Dp^2v^{2g+2} = v^{2g+2}f(u/v) \in \mathbb{F}_q(t)[u,v]$$ which is a degree $2g+2$ homogeneous polynomial in $u,v$ and is denoted by $F(u,v)$ from now on.
	
	Thus $D$ is the squarefree part of $F(u,v)$ in $\mathbb{F}_q[t]^* / (\mathbb{F}_q[t]^*)^2$ for some $u,v \in \mathbb{F}_q[t]$ and we can give a bound on the number of $D$ by estimating the number of squarefree values taken by $F(u,v)$.
	
	The main tool in our lower bound is a theorem of Poonen showing that squarefree polynomials over a localization of a polynomial ring take many squarefree values.
	
	\begin{Proposition}\label{poonen}(Poonen~\cite{Poonen})\\
		Let $P$ be a finite set of primes in $\mathbb{F}_q[t]$, $A$ be the localization of $\mathbb{F}_q[t]$ by inverting the primes in $P$, $K = \mathbb{F}_q(t)$, $f \in A[x_1, \ldots , x_m]$ be a polynomial that is squarefree as an element of $K[x_1, \ldots , x_m]$ and for a choice of $x \in \mathbb{F}_q[t]^m$, we say that $f(x)$ is squarefree in $A$ if $(f(x))$ is a product of distinct primes in $A$. For $a \in A$, define $|a| = |A/a|$ and for $a \in A^n$, define $ |a| = \max{|a_i|}$. Let
		\[ S_f := \{ x \in \mathbb{F}_q[t]^m : f(x) \text{ is squarefree in } A \} \]
		and 
		\[ \mu_{S_f} := \lim\limits_{N \rightarrow \infty} \frac{|\{ a \in S_f : |a|<N \}|}{N^m} \]
		For each nonzero prime $ p$ of $A $, let $c_p$ be the number of $ x \in ( A / p^2 )^m $ that satisfy $f(x) = 0$ in $  A / p^2 $. 
		
		Then the limit  $\mu_{S_f}$ exists and is equal to $ \prod_{p}(1-c_p / |p|^{2m})$.
	\end{Proposition}

	\begin{proof}
		This proposition follows Theorem 8.1 of \cite{Poonen} by setting the "box" to be $\{ u,v \in \mathbb{F}_q[t]: |u|,|v|<N  \}$.
	\end{proof}

	\begin{Remark}\label{defineA}
		 Proposition \ref{poonen} only helps us if $\mu_{S_f}>0$. In order to ensure this, it suffices to check that none of the factors $(1-c_p / |p|^{2m})$ is zero where we take $m=2$ for our case.
		
		If for some prime $\pi$ in $\mathbb{F}_q(t)$, $1-c_\pi / |\pi|^{4}=0$, then this means $F(u,v) \mod \pi^2$ vanishes for all $ (u,v) \in (\mathbb{F}_q[t])^2$. Thus $F(u,v) \mod \pi$ vanishes for all $ (u,v) \in (\mathbb{F}_q[t]/\pi)^2$. Since the coefficients of $F$ are units in $\mathbb{F}_q[t]$, $F \mod \pi$ is not the zero polynomial. This implies it can at most have $\deg F |\mathbb{F}_q[t]/\pi|$ solutions over $\mathbb{F}_q[t]/\pi$. So $$\deg F |\mathbb{F}_q[t]/\pi| \ge |\mathbb{F}_q[t]/\pi|^2$$ which is equivalent to $\deg F \ge |\pi|$.
		
		Thus, we choose $P_f$ be the set of primes $P$ of $\mathbb{F}_q(t)$ such that $|P| < n $. Let $A$ be the localization of $\mathbb{F}_q[t]$ by inverting primes in $P_f$. This implies $1-c_p / |p|^{4} \ne 0$ for any prime $p$ in $A$. So is their product.
	\end{Remark}

	We now have all the tools we need for the proof of Proposition \ref{MainProp}.

	\begin{proof}[Proof of Proposition \ref{MainProp}]
	Let $y^2=f(x)$ be the defining equation for $C_0$ stated in the Proposition. Fix a factorization of $f(x)$ as follows:
	
	If $\deg f = 2g+2$, then by the condition of the proposition, $f$ is squarefree and reducible. Fix a nontrivial factorization $f=f_1f_2$ where $\gcd(f_1,f_2)=1$.
	
	If $\deg f = 2g+1$, let the the factorization $f=f_1f_2$ be the trivial one where $f_1=f$ and $f_2=1$.
		
	Let $n=2g+2$. Given $u,v \in \mathbb{F}_q[t]$, define $F(u,v) = v^{n}f(u/v)$ which is a squarefree homogeneous polynomial in $\mathbb{F}_q[t][u,v]$. Then the factorization of $f$ induces a natural factorization $F=F_1F_2$ where
	 $$F_1(u,v)=v^{\deg f_1}f_1(u/v)$$ $$F_2(u,v)=v^{n - \deg f_1}f_2(u/v).$$
	
	Recall Definition \ref{DefinePN}
	\[P(N) = \{ D \in \mathbb{F}_q[t] : D \text{ monic, squarefree, } |D|<N \}\]
	
	Define set
	$$G(N) = \{ D\in P(N) : \text{curve } y^2 = D \text{ admits a dominant map to } C_0  \}$$
	
	Let the set $A$ be defined as in Remark \ref{defineA}. Suppose $u$ and $v$ are elements of $ \mathbb{F}_q[t]$ such that $F(u,v)$ is squarefree in $A$, take $D \in \mathbb{F}_q[t]$ to be the squarefree part of $F(u,v)$ in $\mathbb{F}_q[t]$; then the curve $Dy^2=f(x)$ has a point $(u/v,a/v^{n/2})$ over $K=\mathbb{F}_q(t)$ where $a$ is a unit in $A$. By Proposition \ref{Tofunctionfield}, this implies curve $y^2 = D$ admits a dominant map to $C_0$.
	
	For pairs $(u,v)$ with $ |u|, |v| < N^{1/n}$, we get $|F(u,v)| < N$.
	
	Thus, we can define a subset of $G(N)$ as:
	 $$G'(N) = \{D \in P(N): \exists u,v \in \mathbb{F}_q[t], |u|,|v|<N^{1/n} \text{ and } F(u,v)=a^2D \}$$
	where $a$ is a unit in $A$.
	
	Define set $W(N)$ as follows: $$W(N) = \{ (u,v) \in \mathbb{F}_q[t]:F(u,v) \text{ squarefree in } A, |u|,|v|<N^{1/n} \}.$$
	
	We have an explicit surjective map from $W(N)$ to $G'(N)$.
	$$ \phi: (u,v) \mapsto D$$
	where $D$ is the squarefree part of $F(u,v)$ in $\mathbb{F}_q[t]$.	
	
	By Proposition \ref{poonen}, 
	\[ \lim\limits_{N \rightarrow \infty}\frac{|W(N)|}{N^{2/ n}} \gg \mu_{S_f} >0 \]
	
	Now to give a lower bound on the size of $G'(N)$, we need to give an upper bound on the size of each fiber of $\phi$. 
	
	For each fixed $D \in G'(N)$, want to count pairs $(u,v)$ with $F(u,v) = a^2 D$ for some unit $a$. Since $F=F_1 F_2$, for each $a$, there exist decompositions $D_1 D_2 = a^2D$ such that 
	\[ F_1(u,v)=D_1\]
	\[ F_2(u,v)=D_2\]
	By construction, we have that $F_1$ and $F_2$ are coprime. So there are fewer than $n^2$ solutions for each pair of equations by Bezout's theorem and there are at most $d(a^2D)$ such decompositions for each $a$ where $d(a^2D)$ denotes the number of factors of $a^2D$ in $\mathbb{F}_q[t]$.
	
	For each $(u,v) \in W(N)$, $|F(u,v)| \le N$. Thus, we can give an upper bound for $d(a^2D)$ by letting $c(N) = \max \{ d(X) : X \in \mathbb{F}_q[t], |X|<N \}$.
	
	For each fixed $D \in G'(N)$, the size of $\phi^{-1}(D)$ is bounded above by $n^2 c(N)$.
	
	Then for any $N$, $$|G'(N)| \geq \frac{|W(N)|}{n^2c(N)}$$
	
	Since $d(X) < |X|^\epsilon$ for any $\epsilon>0$ and $X \in \mathbb{F}_q[t]$ when $|X|$ is sufficiently large, we get 
	$$ |G(N)| \geq |G'(N)| \geq \frac{|W(N)|}{n^2c(N)} \gg \frac{\mu_{S_f}}{c_\epsilon} N^{2/ n - \epsilon} $$
	where $c_\epsilon$ is a constant depending on $\epsilon$.

	\end{proof}
		
	\section{A View Toward Ranks of Elliptic Curves}

	We start by recalling some standard definitions. 
	
	\begin{Definition}
		An elliptic curve $E$ defined over $\mathbb{F}_q(t)$ is constant if it can be defined by a Weierstrass form with coefficients in $\mathbb{F}_q$.
		
		An elliptic curve $E$ defined over $\mathbb{F}_q(t)$ is isotrivial if there is a finite extension $L$ of $\mathbb{F}_q(t)$ such that $E$ becomes constant over $L$. Equivalently, $E$ is isotrivial if and only if $j(E) \in \mathbb{F}_q $.
	\end{Definition}
	
	\begin{Proposition}[Proposition 6.1 of \cite{Ulmer}]
		
		Let $E_0$ be an elliptic curve over $k =\mathbb{F}_q$. Let $K$ be the function field $k(C)$ of a curve $C$ over $k$. Let $E_K = E_0 \times_k K$.
		
		There is a canonical isomorphism $$E_K(K) \cong \Mor_k(C,E_0)$$ where $\Mor_k$ denotes morphisms of $k$-schemes. Under this isomorphism, $E_K(K)_{tor}$ corresponds to the subgroup of constant morphisms.
	\end{Proposition}

	\begin{Corollary}[Corollary 6.2 of \cite{Ulmer}]
		
		Let $J(C)$ be the Jacobian of $C$. Then we have canonical isomorphisms of abelian groups
		$$ E_K(K)/(E_K(K))_{tor} \cong \Hom_{k-av}(J(C),E_0) \cong \Hom_{k-av}(E_0,J(C)) $$
		where $\Hom_{k-av}$ denotes morphisms of abelian varieties over $k$.
	\end{Corollary}

	\begin{Proposition}
	Let $E = E_0 \times \mathbb{F}_q(t)$ be a constant elliptic curve over $\mathbb{F}_q(t)$. For any $D\in \mathbb{F}_q[t]$, let $E_D$ denote the quadratic twist of $E$ by $D$.
	Let $P(N)$ be the set $\{ D\in \mathbb{F}_q[t]: \text{ monic,} \text{ squarefree, } |D|<N \}$ as in Definition \ref{DefinePN}.
	Let $R_m(N)$ be the set $\{ D\in P(N): \rank E_D \ge m \}$.
	
	Then for any $\epsilon >0$, there exist nonzero constants $B_\epsilon$ and $N_\epsilon$ such that $$ |R_2(N)| \ge B_\epsilon N^{1/2 - \epsilon} $$ for any $N>N_\epsilon$.
	
	Moreover, if the rank of $ \End_{\mathbb{F}_q}(E_0)$ is $4$, then we can replace $R_2(N)$ with $R_4(N)$ and the conclusion still holds.
\end{Proposition}

\begin{proof}
	By Prop 3.1, for any $\epsilon > 0$, there exists a nonzero constant $B_\epsilon$ such that at least $B_\epsilon N^{1/2 - \epsilon}$ hyperelliptic curves $y^2 = D$ with $|D|<N$ admit a dominant map to $E_0$ when $N$ is large. 
	
	By Proposition 4.2, Corollary 4.3 and  Poincar\'e complete reducibility \cite{Ulmer}, for such $D$, rank $E_D \geq$ rank $\End(E_0)$.
	
	Since our ground field is of positive characteristic, the endomorphism ring of $E_0$ has rank $2$ or $4$.
\end{proof}

\begin{Proposition}
	Let $E = E_0 \times \mathbb{F}_q(t)$ be a constant elliptic curve over $\mathbb{F}_q(t)$ where $E_0[2](\mathbb{F}_q) \ne O$. When $p \neq 2$, $q \neq 3,9$ and $a^2 - 4q \notin \{ -3,-4,-7 \}$ where $a$ is the trace of geometric Frobenius acting on the Tate module, we have the following.
	
	Let $P(N) = \{ D\in \mathbb{F}_q[t]: \text{ monic,} \text{ squarefree, } |D|<N \}$.
	
	Let $R_m(N)$ be the set $\{ D\in P(N): \text{rank } E_D \ge m \}$.

	Then for any $\epsilon >0$, there exist nonzero constants $B_\epsilon$ and $N_\epsilon$ such that $$|R_4(N)| \geq B_\epsilon N^{1/3 - \epsilon} $$ for any $N>N_\epsilon$.
	
	Moreover, if the rank of $ \End_{\mathbb{F}_q}(E_0)$ is $4$, then we can replace $R_4$ with $R_8$ and the conclusion holds.
\end{Proposition}

\begin{proof}
	By \cite{Jacobian}, we know when the conditions on $p$, $q$, $a^2-4q$ in the statement of the proposition are satisfied, there exists a Jacobian variety isogenous to $E_0 \times E_0$ over $\mathbb{F}_q$. This Jacobian variety corresponds to a genus $2$ curve $C$. We now show that $C$ has a defining equation that satisfies the condition for Proposition 3.1.
	
	Let $y^2 = f(x)$ be a defining equation for $C$ and assume $\deg f = 6$.
	
	Denote the roots of $f$ by $x_1, \ldots, x_6$. Then the $2$-torsion group of $J(C)$ is generated by divisors $(x_1,0)-(x_i,0)$ where $i=2,3,4,5$. Thus, using this basis, from the action on $x_1, \ldots, x_6$, we get the matrix representation of the Frobenius action on $J(C)[2] \simeq (\mathbb{F}_2)^4$. If Frobenius acts on the roots transitively, then the characteristic polynomial of the action on the $\mathbb{F}_2$ vector space is $x^4+x^2+1$.
	
	Since we know the characteristic polynomial of Frobenius acting on the Tate module of $J(C)$ is $(1-ax+qx^2)^2$, we see that the action of Frobenius on $J(C_0)[2]$ has characteristic polynomial $x^4+1$ when $a$ is even. And $a$ is even if and only if condition $E_0[2](\mathbb{F}_q) \ne O$ holds. Thus Frobenius doesn't act transitively on the Weierstrass points of curve $C$.
	
	Since Frobenius doesn't act transitively on the Weierstrass points of $C$, it has a defining equation of the form $y^2=f(x)$ where $f$ is not irreducible. 
	
	By Prop 3.1, for any $\epsilon>0$, at least $B_\epsilon N^{1/3-\epsilon}$ hyperelliptic curves $y^2 = D$ with $|D|<N$ admit a dominant map to this fixed genus $2$ curve when $N$ is sufficiently large. 
	
	For these $D$,  rank $E_D \geq$ $2$ rank $End(E_0)$.
	
	Since our ground field is of positive characteristic, $\End(E_0)$ has rank $2$ or $4$ which gives rank $4$ and $8$ for quadratic twists $E_D$.
\end{proof}

	From the geometric interpretation of L-functions, we know $L(1/2,\chi_D)=0$ if and only if there exists a dominant map from the hyperelliptic curve defined by $C: y^2=D$ to $E_0$. Using the statements above, this condition is equivalent to the quadratic twist of the constant elliptic curve $E_0 \times \mathbb{F}_q(t)$ by $D$ having positive rank.
	
	More precisely, we have the following proposition.

	\begin{Proposition}
		\label{prop:rankgivezeros}
		For $q$ a square, let $E_0$ be an elliptic curve which admits $\sqrt{q}$ as a Frobenius eigenvalue.
		
		Denote by $E$ the base change of $E_0$ to the function field $\mathbb{F}_q(t)$.
		
		Denote by $E_D$ be the quadratic twist of $E$ by $D$ where $D$ is a squarefree polynomial in $\mathbb{F}_q[t]$.
		
		Recall the definitions
		\begin{align*}
		P(N)& = \{ D \in \mathbb{F}_q[t] : D \text{ monic, squarefree, } |D|<N \}\\
		g(N)& = \{ D \in P(N) : L(1/2, \chi_D) = 0 \}
		\end{align*} 
		
		Let $R_m(N)$ be the set $\{ D\in P(N): \text{rank } E_D \ge m \}$.
		
		Then $g(N)= R_2(N)$.

	\end{Proposition}

	\begin{proof}
		Let $D$ be a monic, squarefree polynomial in $\mathbb{F}_q[t]$ and $C$ the hyperelliptic curve defined by $y^2 = D$. Let $K$ be the function field $k(C)$ of $C$ and $E_K=E_0 \times_k K$.
		
		Let $J(C)$ be the Jacobian of $C$. Assume that $J(C)$ is isogenous to $E_0^m \times A$ over $k$ where $A$ is an Abelian variety admitting no nonzero morphisms to $E_0$.
		
		Then by Corollary 4.3, $E_K(K)/E_K(K)_{tor} \simeq (\End(E_0))^m$. Since $E_0$ is defined over $\mathbb{F}_q$, rank $\End(E_0)$ is at least 2; we conclude that rank $E_K \ge 2m$.
		
		 We have rank $E_K(K)$ $=$ rank $E_D(\mathbb{F}_q(t))$ $+$ rank $E(\mathbb{F}_q(t))$. But rank $E(\mathbb{F}_q(t))$ is always $0$ since $\mathbb{F}_q(t)$ is the function field of $\mathbb{P}^1$ and there is no nonconstant map from $\mathbb{P}^1$ to an elliptic curve. 
		 
		 Thus rank $E_D$ $=$ rank $E_K \ge 2m$ .
		 
		 As was studied before, $L(1/2,\chi_D) = 0$ if and only if $C$ admits a dominant map to $E_0$. This is equivalent to $m > 0$ and rank $E_D \geq 2$.
	\end{proof}
	
	Thus, by Prop 4.6, results on quadratic characters can be used to give lower bounds on the number of elliptic curves with rank $\geq 2$ in quadratic twist families of constant elliptic curves.
	
	There are lots of heuristics and results on the study of ranks of elliptic curves in a quadratic twist family  over number fields (\cite{Melanie}, \cite{Mazur}). For example, with a fixed $E/ \mathbb{Q}$, let $d$ range over fundamental discriminants in $\mathbb{Z}$. Define set
	\[ N(X) = \{ d<X: \text{rank}(E_d) \geq 2 \text{ and is even }\} \]
	Then it is conjectured by Sarnak that 
	\[ \left|N(X)\right| = X^{3/4+o(1)} \]
	Following Katz-Sarnak philosophy, Conrey, Keating, Rubinstein, and Snaith ~\cite{CKRS} made the previous conjecture more precise. They conjectured that there exist constants $c_E$, $e_E$ such that
	\[ \left|N(X)\right| = (c_E+o(1))X^{3/4}(\ln(X))^{e_E} \] 
	Gouv\^ea and Mazur~\cite{Gouvea} proved under the parity conjecture, for any $\epsilon>0$, there exists a constant $X_{\epsilon}$ such that for all $X \geq X_{\epsilon}$,
	\[ \left|N(X)\right| > X^{1/2 - \epsilon}  \]
	In the same spirit, Karl Rubin and Alice Silverberg~\cite{Rubin} showed unconditionally that if either
	
	\begin{itemize}
		\item $E[2]$ has a non-trivial Galois equivariant automorphism and $\End_{\mathbb{C}}(E) \ne \mathbb{Z}[i]$, or
		\item $E$ has a rational subgroup of odd prime order $p$ and $\mathbb{Z}[\sqrt{-p}] \nsubseteq \End_{\mathbb{C}}(E)$.
	\end{itemize}
	one has, for $X \gg 1$,
	\[ \left| \left\{ d<X: \text{rank}(E_d) \geq 2  \right\} \right| \gg X^{1/3} \]
	They also showed the existence of a family of elliptic curves $E$ over $\mathbb{Q}$ such that
	\[ \left| \left \{ d<X: \text{rank}(E_d) \geq 3 \right \}\right| \gg X^{1/6} \]
	Goldfeld\cite{Goldfeld} conjectures that the average rank of quadratic twists of an elliptic curve is $1/2$ , to be more precise, 
	\[ \lim\limits_{X \rightarrow \infty}  \frac{\sum_{|d|<X} \text{rank}(E_d)}{\left| \left \{d:|d|<X, \text{squarefree} \right \}\right|} = \frac{1}{2} \]
	What underlies this conjecture is a widely held belief that $50 \%$ of the elliptic curves have rank $0$ and $50 \%$ have rank $1$. This is a combination of parity principle and minimalist philosophy.
	
	In our case, the parity principle does not apply, since all the quadratic twists in our family have even rank. 
	
	Thus, we don't expect the average rank of this family to approach $1/2$. But still, we would expect minimalist philosophy which means $0 \%$ of elliptic curves in this family have rank $\ge 2$. 
	
	And this expectation is supported by Bui and Florea's result mentioned in the first section for the odd degree case.
	
	\begin{Corollary}
		For $q$ a square, let $E$ be an elliptic curve over $\mathbb{F}_q$ where $$L(s,E) = 1-sq^{1/2-s} + q^{1-2s}.$$
		Let $$P'(g) = \{ D\in \mathbb{F}_q[t]: \text{ monic,} \text{ squarefree,} \text{ of odd degree, } \deg D \leq 2g+1 \}.$$
		$$R'(g) = \{ D\in P'(g): E_D \text{ has rank } 0 \}.$$
		Then $$ \lim\limits_{g \rightarrow \infty} \frac{|R'(g)|}{|P'(g)|} \geq 0.9427 \cdots + o(1) .$$
	\end{Corollary}

	\begin{proof}
		This follows from Prop~\ref{prop:rankgivezeros} and Corollary 2.1 of \cite{BF}.
	\end{proof}
	
	\section{Proof of the Main Theorem}
	
	In this section, we will use Prop 3.1 as our main tool to prove the three statements of Theorem 1.2.
	
	\begin{proof}[Proof of Theorem 1.2 (1)]
		Following the theorem of Honda--Tate, when $q$ is a square, the simple Abelian varieties defined over $\mathbb{F}_q$ with $q^{1/2}$ being a Frobenius eigenvalue are elliptic curves. We will pick one such curve and call it $E$ with a Weierstrass form. 
		
		When $q$ is a square, $C : y^2 = D$ admits a dominant map to $E$ if and only if $L(1/2, \chi_D) = 0$.
		
		By Proposition 3.1, since $E$ has genus $1$ with an odd defining equation, for any $\epsilon>0$, there are at least $B_\epsilon N^{1/2-\epsilon}$ polynomials with $|D|<N$ satisfying the condition where $B_\epsilon$ is a nonzero constant.
		
		So we get for polynomials $D \in \mathbb{F}_q[t]$ with $|D|<N$, for any $\epsilon>0$, there are at least $B_\epsilon N^{1/2-\epsilon}$ which have the property that $L(1/2, \chi_D) = 0$ for $N$ large.
	
	\end{proof}

	\begin{proof}[Proof of Theorem 1.2 (2)]
		When $q$ is not a square, the simple $\mathbb{F}_q$ Abelian varieties with $q^{1/2}$ as a Frobenius eigenvalue form an isogeny class of Abelian surfaces. They are exactly the Weil restriction of scalars of the class of elliptic curves defined over $\mathbb{F}_{q^2}$ which have $q$ as a Frobenius eigenvalue. 
		
		By results of Howe, Nart and Ritzenthaler~\cite{Jacobian}, for all $q>3$, there is an abelian variety $A_q$ having $\sqrt{q}$ as a Frobenius eigenvalue which is the Jacobian of a smooth genus-2 curve. It will play the same role in this section as the elliptic curve $E$ for the case when $q$ is a square.
		
		Now in this case, we still have that for a polynomial $D \in \mathbb{F}_q[x]$ to have $L(1/2, \chi_D) = 0$, $A_q$ is isogenous over $\mathbb{F}_q$ to a subabelian variety of the Jacobian of curve $C$ given by $y^2 = D$.  
		
		Unlike the previous case, a map $J(C) \rightarrow A_q$ won't induce a map from $C$ to $C_0$. 
		
		However, the existence of a map $C \rightarrow C_0$ would guarantee $J(C_0) = A_q$ to be isogenous to a subabelian variety of $J(C)$. 
		
		In order to use Prop 3.1, we need $C_0$ to have a defining equation of the form $y^2=f(x)$ where $\deg f = 6$ and $f$ is reducible.
		
		We will show that $C_0$ has such an equation for all $q$; that is, for each $q$ and each $C_0$ whose Jacobian is isogenous to $A_q$, the $q$-th Frobenius doesn't act transitively on the Weierstrass points of $C_0$. 
		
		Denote the roots of $f$ by $x_1, \ldots, x_6$. Then the $2$-torsion group of $J(C_0)$ is generated by divisors $(x_1,0)-(x_i,0)$ where $i=2,3,4,5$. Thus, using this basis, from the action on $x_1, \ldots, x_6$, we get the matrix representation of the Frobenius action on $J(C_0)[2] \simeq (\mathbb{F}_2)^4$. If Frobenius acts on the roots transitively, then the characteristic polynomial of the action on the $\mathbb{F}_2$ vector space is $x^4+x^2+1$.
		
		Since we know the characteristic polynomial of Frobenius acting on the Tate module of $J(C_0)$ is $x^4-2qx^2+q^4$, we see that the action of Frobenius on $J(C_0)[2]$ has characteristic polynomial $x^4+1$. Thus Frobenius doesn't act transitively on the Weierstrass points.
		
		Since Frobenius doesn't act transitively on the Weierstrass points of $C_0$, it has a defining equation of the form $y^2=f(x)$ where $f$ is not irreducible. By applying Proposition 3.1, for any $\epsilon>0$, there are at least $B_\epsilon N^{1/3-\epsilon}$ polynomials with $|D|<N$ with the curve defined by $y^2=D$ admitting a dominant map to $C_0$ where $B_\epsilon$ is a nonzero constant.
		
		We thus conclude that $g(N)$ is at least $B_\epsilon N^{1/3-\epsilon}$ for $N$ large.
	\end{proof}

	\begin{proof}[Proof of Theorem 1.2(3)]
		We used Magma to go through all hyperelliptic curves defined by monic squarefree polynomial over $\mathbb{F}_3$ and found that the curve $C$ defined by $y^2 = x(x^8-1)$ admits $\sqrt{3}$ as a Frobenius eigenvalue. Since $C$ has an odd defining equation and is of genus $4$, by applying Proposition 3.1, we conclude for any $\epsilon>0$, at least $B_\epsilon N^{1/5-\epsilon}$ hyperelliptic curves admit a dominant map to $C$ where $B_\epsilon$ is a nonzero constant and $N$ is large. 
	\end{proof}
	
	\section{Data and Remarks}
	To get a direct view of our main problem, we used Magma to list all monic squarefree polynomials up to a certain degree over some finite fields and evaluate the L-functions corresponding to the hyperelliptic curves defined by these polynomials at the central point to get a count on the ones with value $0$. We have listed the data over fields $\mathbb{F}_5$ and $\mathbb{F}_9$ in the following tables. For field $\mathbb{F}_3$, there was only one curve of genus $4$ given by a degree $9$ defining equation found during the enumeration for polynomials of degree up to $12$.
	
	In the following tables, the first column is the degree $d$ of polynomials. Second column is the number of polynomials of degree $d$ whose corresponding L-function vanishes at $s=1/2$. The set of such polynomials is denoted as $g'(q^d)$. Note that the set $g(q^d)$ studied in the paper is the union of $g'(q^k)$ for all $k \le d$. The third column lists the total number of degree $d$ monic squarefree polynomials. The last column is the value ${log(g'(q^d))}/{log(q^d-q^{d-1})}$. By our main theorem, it has a $\liminf$ of at least $1/3$ for $\mathbb{F}_5$ and $1/2$ for $\mathbb{F}_9$ as $d \rightarrow \infty$.

		\begin{center}
	\begin{tabular}{  |c|r|c|c|}
		\hline
		\multicolumn{4}{|c|}{$\mathbb{F}_5$}\\
		\hline
		Degree $d$ &$|g'(5^d)|$ &$5^d-5^{d-1}$ & $\frac{log(|g'(5^d)|)}{log(5^d-5^{d-1})}$ \\
		\hline
		3  &  0& 100 & \\
		4  & 0& 500 &  \\
		5  & 1&2500&0  \\
		6  & 0&12500&  \\
		7  & 10&62500& 0.2085 \\
		8  & 5&312500&  0.1272  \\
		\hline
	\end{tabular}
\end{center}

For degree $9$ and $10$, due to the large number of monic squarefree polynomials, we randomly sampled 5000000 data points for each and got the following data.  The sample set is denoted by $S$. If we estimate the density $|g'(5^d)|/(5^d-5^{d-1})$ to be equal to the same density $|S \cap g'(5^d)|/|S|$, then we get an approximation for $\frac{log(|g'(5^d)|)}{log(5^d-5^{d-1})}$ which was put in the last column.

	\begin{center}
	\begin{tabular}{  |c|r|c|c|}
		\hline
		\multicolumn{4}{|c|}{$\mathbb{F}_5$}\\
		\hline
		Degree $d$ &$ |S \cap g'(5^d)|$ & $|S|$ & $\frac{log(|g'(5^d)|)}{log(5^d-5^{d-1})}$ \\
		\hline
		9  & 317&5000000& 0.3222 \\
		10 & 89&5000000&0.3109  \\
		\hline
	\end{tabular}
\end{center}

Over $\mathbb{F}_5$, we see there exists a genus $2$ curve defined by a degree $5$ polynomial with Frobenius eigenvalue $\sqrt{5}$. This polynomial is $x(x^4-1)$. Unlike hyperelliptic curves defined over larger fields, this curve doesn't have an even degree model. That explains why there is no quadratic character with conductor $5^6$ whose L-function vanishes at $s=1/2$.

		\begin{center}
	\begin{tabular}{  |c|r|c|c| }
		\hline
		\multicolumn{4}{|c|}{$\mathbb{F}_9$}\\
		\hline
		Degree $d$ &$|g'(9^d)|$ &$9^d-9^{d-1}$ & $\frac{log(|g'(9^d)|)}{log(9^d-9^{d-1})}$ \\
		\hline
		3  & 6    &  648 & 0.2768\\
		4  & 18    &  5832 & 0.3333 \\
		5  & 216   &  52488  &0.4946 \\
		6  & 180   &  472392&  0.3975 \\
		7  & 8658 &   4251528 & 0.5940  \\
		\hline
	\end{tabular}
\end{center}

Similarly,for degree $8$, $9$ and $10$, 5000000 data points for each were randomly sampled and we got the following data. The last column is the approximation gotten the same way as the case of field $\mathbb{F}_5$ listed above.

\begin{center}
	\begin{tabular}{  |c|r|c|c|}
		\hline
		\multicolumn{4}{|c|}{$\mathbb{F}_9$}\\
		\hline
		Degree $d$ &$ |S \cap g'(9^d)|$ & $|S|$ & $\frac{log(|g'(9^d)|)}{log(9^d-9^{d-1})}$\\
		\hline
		8 & 2660&  5000000&  0.5682 \\
		9 &3262&  5000000 & 0.6269 \\
		10 & 532 & 5000000 &  0.5814\\
		\hline
	\end{tabular}
\end{center}

From this table, we can see over $\mathbb{F}_9$, characters defined by odd degree polynomials are more likely to have their L-function vanish at $s=1/2$.	Thus, what this data tells us is that hyperelliptic curves defined over $\mathbb{F}_{p^2}$ with a Frobenius eigenvalue $p$ is more likely to have a rational Weierstrass point.

One explanation for this phenomenon is the observation that elliptic curves defined over $\mathbb{F}_{p^2}$ with Frobenius eigenvalues $p$ and $p$ have full $2$ torsion group over $\mathbb{F}_{p^2}$. This is because the $p$th Frobenius acts on Tate module $T_l$ by multiplication by $p$; thus, if $p \equiv 1 \mod l$ then the action is trivial on $E[l]$. And this is equivalent to $l$-torsion points being defined over the ground field.

Thus the elliptic curve $E$ we used in the proof of Theorem 1.1 part 1 is defined by $y^2 = x(x-1)(x-\lambda)$ where $\lambda \in \mathbb{F}_q$. And the hyperelliptic curves $C$ which admit a dominant map to $E$ have defining equations of the form $y^2 = F(x) = u(x)(u(x)-v(x))(u(x)-\lambda v(x))v(x)$.

For $C$ to have a rational Weierstrass point is equivalent to $F(x)$ having a rational root. As we can see, instead of being a random polynomial over $\mathbb{F}_q$, $F(x)$ admits a factorization into four factors; this should increase the likelihood of its having an $\mathbb{F}_q$ rational root.

\end{document}